\documentstyle[amssymb,amsfonts,12pt]{amsart}
\newtheorem{theorem}{Theorem}[section]
\newtheorem{lemma}[theorem]{Lemma}

\newtheorem{proposition}[theorem]{Proposition}

\theoremstyle{remark}
\newtheorem{remark}[theorem]{Remark}

\def\C{{\mbox{\rm\kern.24em
\vrule width.03em height1.43ex depth-.052ex \kern-.26em C}}}
\def\QSet{\mbox{\rm\kern.24em
\vrule width.03em height1.48ex depth-.051ex \kern-.26em Q}}
\def\Z{{\bf Z}}
\def\R{{\mbox{\rm I\kern-.22em R}}}
\def\E{{\mbox{\rm I\kern-.22em E}}}
\def\P{{\bf P}}

\def\\P{{\mathcal P}}

\def\bas{\begin{align*}}
\def\eas{\end{align*}}
\def\bi{\begin{itemize}}
\def\ei{\end{itemize}}
\newenvironment{proof}{\noindent {\bf Proof} }{\endprf\par}
\def \endprf{\hfill  {\vrule height6pt width6pt depth0pt}\medskip}
\def\emph#1{{\it #1}}

\begin{document}
  \title[Divergence of  combinatorial averages]{Divergence of  combinatorial averages and the unboundedness of the trilinear Hilbert transform}

\author[C. Demeter]{Ciprian Demeter}
  \address[C. Demeter]{ Department
  of Mathematics\\ UCLA\\ Los Angeles CA 90095-1555} \email{demeter@@math.ucla.edu}

\begin{abstract}
We consider multilinear averages in ergodic theory and harmonic analysis and prove their divergence in some range of $L^p$ spaces. This  contrasts with the positive behavior exhibited by these averages in a different range, as proved in \cite{DTT}. We also prove that the trilinear Hilbert transform is unbounded in a similar range of $L^p$ spaces. The underlying principle behind these constructions is stated, setting up the stage for more general results.
\end{abstract}
\maketitle

\section{Introduction}
\label{sec:1}

Multilinear ergodic averages have proved to be a powerful tool in settling problems from combinatorial number theory. This strategy was  initiated by Furstenberg in \cite{Fu}. He gave an ergodic theoretical proof of  a theorem of Szemer\'edi about the existence of arbitrarily long arithmetic progressions in subsets of integers with positive upper density. 

Let ${\bf X}=(X, \Sigma, m, T)$ be a dynamical system, i.e. a complete probability space $(X,\Sigma,m)$ endowed with an invertible  transformation $T:X\to X$ such that $m{T^{-1}}=m$. Furstenberg's method consisted of first proving that
$$\liminf_{N-M\to\infty}\frac{1}{N-M}\sum_{l=M+1}^{N}\int f_1(T^lx)f_2(T^{2l}x)\cdots f_{n}(T^{nl}x)\,dm(x)>0,$$
for each positive nonzero functions $f_i\in L^{\infty}({\bf X})$, and then using a corresponding principle to transfer information to the integers.
Subsequently, other combinatorial averages have been investigated  by various authors.  

Let now $n>1$, $m\ge 1$ and consider the  $(n-1)\times {m}$ matrix $A={(a_{i,j})}_{i=1\:j=1}^{n-1\:m}$ with integer entries. We will consider the averages

 \begin{equation}
\label{eq:averg}
\frac{1}{N^{m}}\sum_{1\le n_1,\ldots,n_m\le N}\prod_{i=1}^{n-1}f_i(T^{\sum_{j=1}^{m}a_{i,j}n_j}x),  
\end{equation}
where $f_1,\ldots,f_{n-1}$ are measurable functions on $X$.
When $m=1$ and $a_{i,1}=i$, we obtain the aforementioned Furstenberg's averages
\begin{equation}
\label{eq:Furs}
\frac{1}{N}\sum_{1\le l\le N}\prod_{i=1}^{n-1}f_i(T^{il}x).
\end{equation}

A  related object of interest is given by the averages on $m$ dimensional cubes. Some version of them played a key role in Gowers'  proof of  Szemer\'edi's theorem \cite{Go}. They correspond to  the case where $n=2^m$ and $A$ is a matrix having on each column a vector from   $V_m=\{0,1\}^m\setminus\bf{0}$. More precisely,  if $\epsilon_1,\ldots,\epsilon_{2^{m}-1}$ is the enumeration of $V_m$ in lexicographically increasing order, then the averages on $m$ dimensional cubes are\begin{equation}
\label{eq:cubes}
\tilde{K}_{m}(f_{\epsilon_1},\dots,f_{\epsilon_{2^m-1}},N)(x)=
\frac{1}{N^{m}}\sum_{\vec{i}\in \{1,2,\ldots N\}^m}\prod_{\epsilon\in V_m}f_{\epsilon}(T^{\vec{i}\cdot\epsilon}x),
\end{equation}  
where $\vec{i}\cdot\epsilon$ denotes the usual dot product.

Both  the averages in ~\eqref{eq:Furs} and those in ~\eqref{eq:cubes} have been shown to converge  in the $L^2$-norm when all the functions $f_{\epsilon}$ are in $L^{\infty}({\bf X})$, a result  due to Host and Kra \cite{KH} (a different  proof for the  averages in ~\eqref{eq:Furs} is due to Ziegler \cite{Z}).

As for the almost everywhere convergence of ~\eqref{eq:averg}, there are two interesting lines of investigation. On the one hand, it is legitimate to ask whether convergence holds for $L^{\infty}$ functions. Bourgain  proved in \cite{Bo2} using Fourier analysis on the torus that the averages in ~\eqref{eq:Furs} converge almost everywhere for $L^{\infty}$ functions, in the case $n=3$. For larger $n$ only partial results are known, for transformations $T$ with nice spectral properties \cite{As:2}, \cite{Le:1}. Assani \cite{As:1} used the semi-norms  introduced by Host and Kra in \cite{KH} to prove the almost everywhere convergence of the averages on cubes for arbitrary $m$, again for bounded functions. 

Once pointwise convergence is established  for a dense class of functions (like $L^{\infty}({\bf X})$), the extension of this convergence to other $L^p$ spaces becomes synonymous with the existence of a maximal inequality for the associated maximal operator. In \cite{DTT},  a very general result is proved, showing the boundedness of the maximal operator
\begin{equation*}
T_{A,{\bf X}}^{*}(f_1,\ldots,f_{n-1})(x)=\sup_{N\ge 1}\frac{1}{N^{m}}\sum_{|n_1|,\ldots,|n_m|\le N }\prod_{i=1}^{n-1}|f_i(T^{\sum_{j=1}^{m}a_{i,j}n_j}x)|
\end{equation*}
in some range of $L^p$ spaces depending on the rank properties of the extended matrix $\E(A)$ defined as
$$\E(A)=\begin{pmatrix}
a_{1,1}&a_{1,2}&\dots & a_{1,m}&1\\
a_{2,1}&a_{2,2}&\dots & a_{2,m}&1\\
\hdotsfor[2.0]4\\
a_{n-1,1}&a_{n-1,2}&\dots & a_{n-1,m}&1\\
0&0&\dots &0&1
\end{pmatrix}.
$$ Our main objective here is to contrast this result with some negative ones, showing that divergence may occur for some functions in $L^p$, with $p$ sufficiently close to 1, due to the maximal operator above failing to be bounded. One of the important aspects of these counterexamples is that they give an explanation to  why the time frequency methods used in \cite{DTT} can not be used all the way up to $L^1$.

One can also look at the averages in ~\eqref{eq:averg} from a different angle and analyze their real variable analog
\begin{equation}
\label{eq:avhar}
\frac{1}{\varepsilon^{m}}\int_{|t_1|,\ldots,|t_m|\le \varepsilon}\prod_{i=1}^{n-1}f_i(x+\sum_{j=1}^{m}a_{i,j}t_j)d\vec{t},  
\end{equation} 
where $f_1,\ldots,f_{n-1}$ are arbitrary measurable functions on $\R$. The question of norm convergence for bounded functions is trivial in this context (one uses approximation and the immediate result for the span of the characteristic functions of intervals). However, the study of the pointwise convergence of the averages in ~\eqref{eq:averg} is intimately connected to  that of the averages in ~\eqref{eq:avhar}. Indeed, a maximal inequality (or the failure of it) in one setting is equivalent to the existence (or failure) of a maximal inequality for the corresponding maximal operator in the other context. The transition back and forth between the reals and a dynamical system is mediated by the integers, which can be regarded as a copy of each individual orbit sitting naturally inside the reals. We will do all the constructions in the real variable setting and then transfer them to the averages in ~\eqref{eq:averg}.

More interestingly, our counterexamples also apply with no modification to the case of the corresponding singular integral operators. In particular, we show that the trilinear Hilbert transform is unbounded in some range of exponents.

Most of the results from the following section have been announced in \cite{DTT}.
\section{Main results}

We start with the divergence of Furstenberg's nonconventional averages \eqref{eq:Furs} for functions in $L^p$ spaces, with $p$ close to $L^1$. We will then sketch the more delicate argument for the averages on cubes.
\begin{theorem}
\label{tt:1}
Define $p_0=1+\frac{\log_6 2}{1+\log_6 2}$ and consider $p<p_0$. In every ergodic dynamical system ${\bf X}=(X, \Sigma,\mu, T)$ there are three  functions $F,G,H\in L^{p}({\bf X})$ such that
$$\limsup_{N\to\infty}\frac{1}{N}\sum_{n=1}^{N}F(T^nx)G(T^{2n}x)H(T^{3n}x)=\infty$$
for $\mu$ a.e. $x\in X$.
\end{theorem}
\begin{proof}
The idea of the proof is to construct ``small'' subsets of integers $A_0$ and $B_0$, such that $2B_0-A_0$ is also ``small'' while $2A_0-B_0$ is ``large''. This choice is dictated by the geometry of the averages under investigation, as explained below.

For a fixed $k\ge 1$ consider the following subsets of $[-1,1]$:
$$A=\{\sum_{i=1}^{k}\frac{a_i}{12^i}+z:a_i\in A_0:=\{-4,-2,0\},0\le z\le \frac{1}{2\times12^k}\}$$
$$B=\{\sum_{i=1}^{k}\frac{b_i}{12^i}+z:b_i\in B_0:=\{0,1,2,3\},0\le z\le \frac{1}{2\times12^k}\}$$
\begin{align*}
C&=\{\sum_{i=1}^{k}\frac{2b_i-a_i}{12^i}+z:a_i\in\{-4,-2,0\},b_i\in\{0,1,2,3\},0\le z \le \frac{1}{2\times12^k}\}\\&=\{\sum_{i=1}^{k}\frac{c_i}{12^i}+z:c_i\in\{0,2,4,6,8,10\},0\le z\le \frac{1}{2\times12^k}\}
\end{align*}
\begin{align*}
D&=\{\sum_{i=1}^{k}\frac{2a_i-b_i}{12^i}+z:a_i\in\{-4,-2,0\},b_i\in\{0,1,2,3\},0\le z \le \frac{1}{8\times12^k}\}\\&=\{\sum_{i=1}^{k}\frac{d_i}{12^i}+z:d_i\in\{-11,-10,\ldots,-1,0\},0\le z\le \frac{1}{8\times12^k}\}.
\end{align*}
We claim that 
\begin{equation}
\label{neweqH3}
m\left\{x\in [-1,0]:\int_0^11_A(x+t)1_B(x+2t)1_C(x+3t)dt\ge \frac{1}{8\times12^k}\right\}\ge \frac{1}{8},
\end{equation}
 where $m$ denotes the Lebesgue measure. Indeed, for each $x\in D$ of the form $x=\sum_{i=1}^{k}\frac{2a_i-b_i}{12^i}+z$ (for some $i$), with $0\le z\le \frac{1}{8\times12^k}$, and for each $t$ of the form $t=\sum_{i=1}^{k}\frac{b_i-a_i}{12^i}+z'$ (for the same $i$), with $0\le z'\le \frac{1}{8\times12^k}$, we immediately see that $x+t\in A$, $x+2t\in B$ and $x+3t\in C$. It now suffices to note that $m(D)=\frac{1}{8}$, and the claim follows. From here we deduce, by the definition of Riemann integral, that for some sufficiently large $N_k$ (independent of $x$) we have 
\begin{equation}
\label{eq:meas1}
m\left\{x\in [-1,0]:\frac{1}{N_k}\sum_{n=1}^{N_k}1_A(x+\frac{n}{N_k})1_B(x+\frac{2n}{N_k})1_C(x+\frac{3n}{N_k})\ge \frac{1}{16\times12^k}\right\}\ge \frac{1}{8}.
\end{equation}
Now consider a $p<p_0$ and the dynamical system ${\bf X_k}=([-1,1],B, m_1, T_k)$, where $B$ is the restriction of the Lebesgue algebra, $m_1$ is the normalized Lebesgue measure (i.e. $m_1([-1,1])=1$)and  $T_k(x)=x+\frac{1}{N_k}$, with the addition considered modulo the interval $[-1,1]$. We can now rephrase ~\eqref{eq:meas1} as 
\begin{equation}
\label{eq:meas2}
m_1\left\{x\in [-1,1]:\sup_{N\ge 1}\frac{1}{N}\sum_{n=1}^{N}1_A(T_k^nx)1_B(T_k^{2n}x)1_C(T_k^{3n}x)\ge \frac{1}{16\times12^k}\right\}\ge \frac{1}{16}.
\end{equation}
A more useful way of stating this, upon noting that $\|1_A\|_{L^p(\bf X_k)}=\frac{1}{(4\times 4^k)^{1/p}}$,  $\|1_B\|_{L^p(\bf X_k)}=\frac{1}{(4\times 3^k)^{1/p}}$ and  $\|1_C\|_{L^p(\bf X_k)}=\frac{1}{(4\times 2^k)^{1/p}}$, is 
\begin{equation*}
\label{eq:meas3}
\sup_{\|f\|_{L^p(\bf X_k)}=\frac{1}{(4\times 4^k)^{1/p}}\atop{{\|g\|_{L^p(\bf X_k)}=\frac{1}{(4\times 3^k)^{1/p}}}\atop{\|h\|_{L^p(\bf X_k)}=\frac{1}{(4\times 2^k)^{1/p}}}}}m_1\left\{x\in [-1,1]:\sup_{N\ge 1}\frac{1}{N}\sum_{n=1}^{N}f(T_k^nx)g(T_k^{2n}x)h(T_k^{3n}x)\ge \frac{1}{16\times12^k}\right\}\ge \frac{1}{16}.\end{equation*}
By using the transference principle proved  in the Appendix, it follows that there exist $f_k\in L^p({\bf X}),g_k\in L^p({\bf X})$ and $h_k\in L^p({\bf X})$ with $\|f_k\|_{L^p(\bf X)}=\frac{1}{(4\times 4^k)^{1/p}}, \|g_k\|_{L^p(\bf X)}=\frac{1}{(4\times 3^k)^{1/p}}, \|h_k\|_{L^p(\bf X)}=\frac{1}{(4\times 2^k)^{1/p}}$, such that 
\begin{equation}
\label{eq:meas4}
\mu\left(E_k:=\left\{x\in X:\sup_{N\ge 1}\frac{1}{N}\sum_{n=1}^{N}f_k(T^nx)g_k(T^{2n}x)h_k(T^{3n}x)\ge \frac{1}{32\times12^k}\right\}\right)\ge \frac{1}{16}.\end{equation}
Define now 
$$f=\sum_{k=1}^{\infty}\frac{(4\times 4^k)^{1/p}}{k^2}f_k$$
$$g=\sum_{k=1}^{\infty}\frac{(4\times 3^k)^{1/p}}{k^2}g_k$$
$$h=\sum_{k=1}^{\infty}\frac{(4\times 2^k)^{1/p}}{k^2}h_k.$$ Note that $f,g,h\in L^p(\bf X)$. Moreover, there exists a set $X_0\subset X$ of positive measure such that each $x$ from $X_0$ belongs to infinitely many $E_k$'s. For each such $x$ and $k$ we have 
$$\sup_{N\ge 1}\frac{1}{N}\sum_{n=1}^{N}f(T^nx)g(T^{2n}x)h(T^{3n}x)\ge \frac{(4\times 4^k)^{1/p}(4\times 3^k)^{1/p}(4\times 2^k)^{1/p}}{32k^6\times12^k}\;.
$$
Due to our choice of $p$, the  sequence on the left above goes to $\infty$. Since 
$$\bigcup_{j=1}^{\infty}T^{j}(X_0)=X$$ up to sets of $\mu$ measure 0, it follows that 
$$F:=\sum_{j=1}^{\infty}\frac{1}{j^2}T^jf$$
$$G:=\sum_{j=1}^{\infty}\frac{1}{j^2}T^jg$$
$$H:=\sum_{j=1}^{\infty}\frac{1}{j^2}T^jh$$
satisfy the requirement of Theorem ~\eqref{tt:1}.
\end{proof}
\begin{remark}
No negative result can be proved with these techniques for the bilinear averages, the reason being that $x+t$ and $x+2t$ are linearly independent monomials in \R[x,t]. The result of Lacey \cite {L} shows that they do behave well for some range of indices, however opposite behavior is quite anticipated near $L^1$. 
\end{remark}
\begin{remark}
The index $p_0$ in the above theorem can be pushed as close as desired  to $\frac32$, if one considers instead the more general trilinear ergodic averages 
$$\frac{1}{N}\sum_{n=1}^{N}F(T^{a_1n}x)G(T^{a_2n}x)H(T^{a_3n}x),$$
with $a_1\not=a_2\not=a_3\in\Z$. (the choice for $a_i$ will depend on how close $p_0$ is to $\frac32$)
This follows as  a consequence of the results obtained by  Christ in \cite{CH}, which were also the main inspiration for our investigation here. On the other hand, it is  a consequence  of  Lacey's result \cite {L} and multilinear interpolation, as explained in \cite{DTT}, that $p_0$ can be at most 2. 
\end{remark}
\begin{remark}
The result in Theorem \ref{tt:1} immediately proves the divergence of Furstenberg's averages for any $n\ge 3$ (just choose the remaining functions to be identically equal to one). 
\end{remark}

Interestingly, the constructions above also prove the unboundedness of the trilinear Hilbert transform. This operator is initially defined for piecewise continuous functions $f_i$ with finite support as follows:
$$H_3(f_1,f_3,f_3)(x)=p.v.\int f_1(x+t)f_2(x+2t)f_3(x+3t)\,\frac{dt}{t}$$
The bilinear version $H_2$ of this -also known as the bilinear Hilbert transform $$H_2(f_1,f_2)(x)=p.v.\int f_1(x+t)f_2(x+2t)\,\frac{dt}{t},$$ 
was proved to be bounded by Lacey and Thiele \cite{LT1}, \cite{LT2} in some range of exponents. The type of time-frequency analysis involved in their argument seems insufficient at the moment to address any positive bounds for $H_3$. 

\begin{theorem}
\label{tt:H3}
Define $p_0=1+\frac{\log_6 2}{1+\log_6 2}$ and consider $p<p_0$. Then the inequality
\begin{equation}
\label{failH3bound}
\|H_3(f_1,f_3,f_3)\|_{p/3}\le C\|f_1\|_{p}\|f_2\|_{p}\|f_3\|_{p}
\end{equation}
fails to hold with a universal constant $C$, independent of $f_i$.
\end{theorem}
\begin{proof}
It suffices to note that with the notation in the proof of Theorem \ref{tt:1} we have the following consequence of \eqref{neweqH3}
$$m\left\{x\in [-1,0]:\int_{-\infty}^{\infty}1_A(x+t)1_B(x+2t)1_C(x+3t)\,\frac{dt}{t}\ge \frac{1}{8\times12^k}\right\}\ge \frac{1}{8}.$$
To see this, observe that since $x$ is negative and since $B$ consists only of positive numbers we have 
$$H_3(1_A,1_B,1_C)(x)=\int_{0}^{\infty}1_A(x+t)1_B(x+2t)1_C(x+3t)\,\frac{dt}{t}\ge \int_{0}^{1}1_A(x+t)1_B(x+2t)1_C(x+3t)\,dt$$
for each $x\in [-1,0]$. Note also that as before $\|1_A\|_{L^p}=\frac{1}{(4\times 4^k)^{1/p}}$,  $\|1_B\|_{L^p}=\frac{1}{(4\times 3^k)^{1/p}}$,  $\|1_C\|_{L^p}=\frac{1}{(4\times 2^k)^{1/p}}$, and $\|H_3(1_A,1_B,1_C)\|_{p/3}\ge \frac{1}{8^{3/p}}\frac{1}{(8\times 12^k)}$. By using $f_1=1_A$,  $f_2=1_B$, $f_3=1_C$, a simple computation now shows that the constant $C$ in \eqref{failH3bound} will go to $\infty$ as $k$ goes to $\infty$. 
 
\end{proof}

The following  analog of Theorem ~\ref{tt:1} holds for the averages on cubes \eqref{eq:cubes}. 
\begin{theorem}
\label{tt:2}
 For each $m\ge 3$, for each $p<\frac{2^{m-1}+1}{m+1}$, given any ergodic dynamical system ${\bf X}=(X, \Sigma,\mu, T)$, there are $2^{m}-1$  functions $F_{\epsilon}\in L^{p}({\bf X})$, $\epsilon\in V_m$, satisfying
$$\limsup_{N\to\infty}\tilde{K}_{m}(F_{\epsilon_1},\dots,F_{\epsilon_{2^m-1}},N)(x)=\infty
$$
for $\mu$ a.e. $x\in X$.
\end{theorem}
\begin{proof}
As before, we first analyze  the analog of the above averages  on $\R$, defined by $$
K_{m}(f_{\epsilon_1},\dots,f_{\epsilon_{2^m-1}},\lambda)(x):=
\frac{1}{{\lambda}^{m}}\int_{0}^{\lambda}\ldots\int_{0}^{\lambda}\prod_{\epsilon\in V_m}F_{\epsilon}(x+\vec{t}\cdot\epsilon)d t_1\ldots dt_{m},$$
where $\vec{t}=(t_1,\ldots,t_m)$.
Fix a $k\ge 1$. Define first
$$\tilde{A}_{(1,\ldots,1_{j-1},0,1_{j+1},\ldots,1)}=\{\sum_{i=1}^{k}\frac{a_i}{(2^{m+1})^i}:a_i\in\{0,2^{j-1}\}\}$$
for each $1\le j\le m$, (which will be denoted with $B_j$, for simplicity  in future reference), and also
$$\tilde{A}_{(1,\ldots,1)}=\{\sum_{i=1}^{k}\frac{a_i}{(2^{m+1})^i}:a_i\in\{0,-\frac{2^{m}}{m-1}\}\},$$ (which will be denoted with $B$, for simplicity  in future reference). All sets $\tilde{A}$ are indexed by vectors in $V_m$. We uniquely determine the remaining $\tilde{A}_{\epsilon}$'s as being the minimal sets which satisfy the constraint: for each $x,t_1,\ldots,t_m\in\R$ satisfying
\begin{equation}
\label{12178}
x+t_1+\ldots+t_{j-1}+t_{j+1}+\ldots+t_m\in  \tilde{A}_{(1,\ldots,1_{j-1},0,1_{j+1},\ldots,1)}
\end{equation}
for each $1\le j\le m$,
and 
\begin{equation}
\label{12179}
x+t_1+\ldots+t_m\in  \tilde{A}_{(1,\ldots,1)},
\end{equation}
we also have that 
$$x+\vec{t}\cdot\epsilon\in \tilde{A}_{\epsilon}$$ for all the remaining $\epsilon\in V_m$. A simple inspection shows that the monomials in $x$ and $t$ from ~\eqref{12178} and~\eqref{12179} are linearly independent. Also, for each $1\le l\le m-2$ and for each $\epsilon\in V_m$ with exactly $l$ nonzero entries,  $\tilde{A}_{\epsilon}$ can be written as a linear combination of exactly $m+1-l$ of the $m+1$ $\tilde{A}_{\epsilon}$'s from ~\eqref{12178} and~\eqref{12179}. For example 
$$\tilde{A}_{(1,0,\ldots,0)}=B_1+B_3+\ldots+B_m-(m-2)B,$$
$$\tilde{A}_{(1,1,0,\ldots,0)}=B_1+B_4+\ldots+B_m-(m-3)B,\;\text{etc.}$$
Since each $B_j$ and $B$ have  $2^k$ elements, we conclude that for each $1\le l\le m-2$,  $m\choose l$ of the $\tilde{A}_{\epsilon}$'s will have at most $2^{(m-l+1)k}$ elements. For $l=m-1$, we have $m$  such $\tilde{A}_{\epsilon}$ (i.e. the $B_j$'s) with $2^k$ elements each, while for $l=m$ there exists one such set (namely $B$) with $2^k$ elements. Note also that 
~\eqref{12178} and~\eqref{12179} imply that  
\begin{equation}
\label{12180}
x\in \tilde{C}:=B_1+\ldots+B_{m}-(m-1)B
\end{equation}
\begin{equation}
\label{12181}
t_j\in \tilde{C}_j:=B-B_{j+1}
\end{equation}
for each $1\le j\le m$, with the obvious agreement that $B_{m+1}:=B_1$. Also, for each $x=b_1+\ldots+b_m-b\in\tilde{C}$ there exists a unique $m$ tuple $\vec {t}=(t_1,\ldots,t_m)\in \prod_{j=1}^{m}\tilde{C}_j$, namely 
\begin{equation}
\label{12182}
(t_1,\ldots,t_m)=(b-b_2,b-b_3,\ldots,b-b_1),
\end{equation}
 such that~\eqref{12178} and~\eqref{12179} hold.

One can easily see that $$\tilde{C}=\{\sum_{i=1}^{k}\frac{c_i}{(2^{m+1})^i}:c_i\in\{0,1,\ldots,2^{m+1}-1\}\}$$ has  $2^{k(m+1)}$ elements, and each two of them are separated by at least  $\frac{1}{2^{k(m+1)}}$. Define now the sets
$$A_{\epsilon}=\tilde{A}_{\epsilon}+[0,\frac{1}{2^{k(m+1)}}]$$
for each $\epsilon\in V_n$ and  
$$C=\tilde{C}+[0,\frac{1}{(m+1)2^{k(m+1)}}]$$ 

Take now some $x=b_1+\ldots+b_m-b+z\in C $, with $0\le z\le \frac{1}{(m+1)2^{k(m+1)}}$. The above discussion shows that for each $t_j\in b-b_{j+1}+[0,\frac{1}{(m+1)2^{k(m+1)}}]$ we have that for each $\epsilon\in V_m$
$$x+\vec{t}\cdot\epsilon\in A_{\epsilon}.$$ Define the functions 
$$f_{k,\epsilon}=1_{A_{\epsilon}}$$ and note that 
$$\|f_{k,\epsilon}\|_{L^p(\R)}\le \left(\frac{|\tilde{A}_{\epsilon}|}{2^{k(m+1)}}\right)^{\frac{1}{p}},$$ and also that 
$$m\left\{x:\int_{-1}^{0}\ldots\int_{-1}^{0}\prod_{\epsilon\in V_m}f_{k,\epsilon}(x+\vec{t}\cdot\epsilon)d\vec{t}\ge \frac{1}{[(m+1)2^{k(m+1)}]^n}\right\}\ge \frac{2^{k(m+1)}}{(m+1)2^{k(m+1)}}=\frac{1}{m+1}.$$ The argument continues like in Theorem ~\ref{tt:1}, and hence we are guaranteed the  negative result as long as the sequence 
$$ \frac{1}{[(m+1)2^{k(m+1)}]^m}\prod_{\epsilon\in V_m}\left(\frac{2^{k(m+1)}}{|\tilde{A}_{\epsilon}|}\right)^{\frac{1}{p}}$$
diverges to $\infty$. Using the discussion on the sizes of the $\tilde{A}_{\epsilon}$'s from before, it suffices to have 
$$ \frac{1}{[(m+1)2^{k(m+1)}]^m}\left(\prod_{l=1}^{m-2}2^{lk{m\choose l}}2^{km(m+1)}\right)^{\frac{1}{p}}\to\infty.$$
This is easily seen to happen for $p<1+\frac{\sum_{l=1}^{n-2}l{m\choose l}}{m(m+1)}=\frac{2^{m-1}+1}{m+1}$. 
\end{proof}
\begin{remark}
Note that a negative result is produced as soon as one realizes that at least one of those $\tilde{A}_{\epsilon}$'s for which $\epsilon$ has at most $m-2$ nonzero entries, has cardinality considerably less than $2^{k(m+1)}$. 
\end{remark}

As in the case of the bilinear averages analyzed earlier, this type of  constructions can not prove divergence for the averages on squares ($m=2$), but do offer valuable information regarding their degenerate analog
\begin{equation}
\label{dff556} 
\frac{1}{N^2}\sum_{i=1}^{N}\sum_{j=1}^{N}f(T^{2i}x)g(T^{2j}x)h(T^{i+j}x).
\end{equation}
We note first that as a result of multilinear interpolation, the maximal operator associated with these averages maps boundedly $L^{p_1}\times L^{p_2}\times L^{p_3}$ into $L^{p_4'}$, where $\frac1{p_4'}=\frac1p_1+\frac1p_2+\frac1p_3$, whenever $p_4'>\frac12$ and $1<p_1,p_2,p_3$. Indeed, the boundedness is immediate in the case $p_1=p_2=1+\epsilon,\;p_3=\infty$, while the other five permutations can be seen similarly, after some change of variables. This is sharp as the following proposition shows.
\begin{proposition}
 The maximal operator associated with the averages \eqref{dff556} fails to map boundedly $L^{p_1}\times L^{p_2}\times L^{p_3}$ into $L^{p_4'}$, whenever $p_4'<\frac12$. 
\end{proposition}
\begin{proof}
To see this, it suffices to repeat the argument in Theorem \ref{tt:1} for $f=g=h=1_A$ where $A:=\{z\in\R:|z|\le \frac{1}{M}\}$, for large $M$. The maximal inequality will have a constant that will go to $\infty$ as $M$ goes to $\infty$. The details are left to the reader. 
\end{proof}

The above proposition combined with the positive result in \cite{DTT} for the nondegenerate squares in the range $p_4'>\frac25$, makes the important point that the almost everywhere behavior on various $L^p$ spaces ($p<\infty$) of the averages $$ 
\frac{1}{N^2}\sum_{i=1}^{N}\sum_{j=1}^{N}f(T_1^{i}x)g(T_2^{j}x)h(T_3^{i+j}x)$$
associated with general commuting measure preserving transformations $T_1,T_2,T_3$, is very sensitive to the relation between the $T_i$'s. This is unlike the case where $f,g,h\in L^{\infty}({\bf X})$, for which the a.e. convergence holds for any (not necessarily commuting) transformations $T_1,T_2,T_3$ (see \cite{As:1}).

The main ingredient behind all the negative results proved here is the fact that the monomials $x+\sum_{j=1}^{m}a_{i,j}t_j,1\le i\le n-1$ are linearly dependent in $\R[x,t_1,\ldots,t_l]$. This allows for the functions involved to be simultaneously large, hence making the maximal operator big on a relevant set of $x$'s. Similar constructions can be made for various other averages of this type, however it is not clear whether this approach can always be applied. Here is a brief account on what the main difficulty is. Let  $r$ be the smallest number for which one can find $r$ linearly dependent monomials as above, and assume for simplicity that these correspond to $i\in\{1,\ldots,r\}$. Then we can always find at least $r-2$ monomials  among $\sum_{j=1}^{m}a_{i,j}t_j,1\le i\le r$ which are linearly independent in $\R[t_1,\ldots,t_l]$. One type of situation occurs when  only  $r-2$ such monomials exist, say the ones with $i\in\{1,\ldots,r-2\}$. By performing a suitable change of variables and by ignoring the influence of $f_i,i\ge r+1$, it suffices to show that the maximal operator
 
\begin{equation}
\sup_{\epsilon>0}\frac{1}{\epsilon^{r-2}}\int_{|t_1|,\ldots,|t_{r-2}|\le \epsilon}\prod_{i=1}^{r-2}|f_i(x+t_i)|\cdot|f_{r-1}(x+\sum_{i=1}^{r-2}b_{i}t_i)||f_{r}(x+\sum_{i=1}^{r-2}c_{i}t_i)|dt_1\ldots dt_{r-2},
\end{equation}
fails to be bounded in some appropriate range of exponents. The coefficients $b_i,c_i$ are arbitrary integers with the property that the monomial $x+\sum_{i=1}^{r-2}c_{i}t_j$ can be written a linear combination of $x+t_i$, $1\le i\le r-2$ and $x+\sum_{i=1}^{r-2}b_{i}t_i$. In this scenario, the constructions would have to be sensitive to the arithmetics of the $b_i$'s and $c_i$'s. The second type of situation occurs when we can find $r-1$ linearly independent monomials in $\R[t_1,\ldots,t_l]$ (the degenerate case). A similar reasoning shows that in this case it suffices to analyze the 
maximal operator
\begin{equation}
\sup_{\epsilon>0}\frac{1}{\epsilon^{r-1}}\int_{|t_1|,\ldots,|t_{r-1}|\le \epsilon}\prod_{i=1}^{r-1}|f_i(x+t_i)|\cdot|f_{r}(x+\sum_{i=1}^{r-1}b_{i}t_i)|dt_1\ldots dt_{r-1},
\end{equation}
where the coefficients $b_i$ are arbitrary integers with the property that the monomial $x+\sum_{i=1}^{r-2}b_{i}t_i$ can be written a linear combination of $x+t_i$, $1\le i\le r-1$. In this case, the same construction like the one described above for the averages on degenerate squares  will produce divergence results for some $f_i\in L^p,1\le i\le r$, whenever $p<\frac{r}{r-1}$.
\section{Appendix}
 For a complete nonatomic probability space $(X,\Sigma,\mu)$, denote by $C(X)$ the family of all the invertible $\mu$-measure preserving transformations $T$ of  $X$. Equip $C(X)$ with the topology of weak convergence, in which $T_{s}\to T$ if and only if $\mu(T_sA\Delta TA)\to 0$ for each $A\in \Sigma$. If a second complete nonatomic probability space $(Y,F, \nu)$ is present, we will denote by $ C(Y,X)$ the set of all invertible, bimeasurable transformations $\beta:Y\to X$ which take measure $\nu$ to measure $\mu$. The following result is due to Halmos \cite{Ha}, and is the key to the proof of the transference lemma ~\ref{ll:1}:
\begin{theorem}
\label{thm:Halmos}
Given an ergodic dynamical system ${\bf Y}=(Y,F, \nu, S)$ and a complete nonatomic probability space $(X,\Sigma , \mu)$, the set
$$\{\beta S\beta^{-1}, \beta\in  C(Y,X)\}$$
is dense in  the weak topology of $C(X)$.
\end{theorem}

\begin{lemma}[Transference Principle]
\label{ll:1}
For a given  dynamical system ${\bf X_T}=(X,\Sigma, \mu, T)$, $p\ge 1$ and positive constants $a,b,c,\lambda$, define
$$\gamma({\bf X_T},a,b,c,\lambda)=\sup_{\|f\|_{L^p(\bf X)}=a\atop{{\|g\|_{L^p(\bf X)}=b\atop{\|h\|_{L^p(\bf X)}=c}}}}\mu\left\{x\in X:\sup_{N\ge 1}\frac{1}{N}\sum_{n=1}^{N}f(T^nx)g(T^{2n}x)h(T^{3n}x)>\lambda\right\}.$$ Then for any ergodic dynamical system  ${\bf Y_S}=(Y,\mathcal F, \nu, S)$  we have 
$$\gamma({\bf X_T},a,b,c,\lambda)\le \gamma({\bf Y_S},a,b,c,\lambda).$$
\end{lemma}
The first goal here is to show that for each $0\le r\le 1$,
 $$S_{r}=\{\sigma\in  C(X):\gamma({\bf X_{\sigma}},a,b,c,\lambda)\le r\}$$ is closed in the weak topology of $ C(X)$. Let $\sigma_s$ be a net in $ C(X)$ converging weakly to some $\sigma$, such that $\gamma({\bf X_{\sigma_s}},a,b,c,\lambda)\le r$. Take $f,g,h\in L^p({\bf X_T})$ with $\|f\|_{L^p(\bf X_T)}=a$, $\|g\|_{L^p(\bf X_T)}=b$ and  $\|h\|_{L^p(\bf X_T)}=c$ . From the definition of the weak topology, $f\circ\sigma_s\to  f\circ\sigma $ in the norm of $L^p({\bf X_T})$ and similarly for $g$ and $h$. There exists a subnet indexed by $\Z_{+}$, which we will denote by $(\sigma_l)$, such that $f\circ\sigma_l\to  f\circ\sigma $, $g\circ\sigma_l\to  g\circ\sigma$ and $h\circ\sigma_l\to  h\circ\sigma $, in the $L^p$-norm. Choose a subset $X_0\subset X$ of full measure such that $\lim_{l\to \infty}f(\sigma_l x)=f(\sigma x)$ and simultaneously for $g$ and $h$, for each $x\in X_0$. This is easily seen to imply that $$\sup_{N\ge 1}\frac{1}{N}\sum_{n=1}^{N}f(\sigma^nx)g(\sigma^{2n}x)h(\sigma^{3n}x)\le \liminf_{l\to\infty}\sup_{N\ge 1}\frac{1}{N}\sum_{n=1}^{N}f(\sigma_l^nx)g(\sigma_l^{2n}x)h(\sigma_l^{3n}x)$$ for each $x\in X_0$. So 
\begin{align*}
\mu\{x\in X:\sup_{N\ge 1}\frac{1}{N}&\sum_{n=1}^{N}f(\sigma^nx)g(\sigma^{2n}x)h(\sigma^{3n}x)>\lambda\}\\&\le \mu\{x\in X:\liminf_{l\to\infty}\sup_{N\ge 1}\frac{1}{N}\sum_{n=1}^{N}f(\sigma_l^nx)g(\sigma_l^{2n}x)h(\sigma_l^{3n}x)>\lambda\}\\&= \lim_{m\to\infty}\mu\{x\in X:\inf_{l\ge m}\sup_{N\ge 1}\frac{1}{N}\sum_{n=1}^{N}f(\sigma_l^nx)g(\sigma_l^{2n}x)h(\sigma_l^{3n}x)>\lambda \}\\&\le \lim_{m\to\infty}\mu\{x\in X:\sup_{N\ge 1}\frac{1}{N}\sum_{n=1}^{N}f(\sigma_m^nx)g(\sigma_m^{2n}x)h(\sigma_m^{3n}x)>\lambda\}\\&\le r.
\end{align*}
The fact that  $S_r$ is closed  follows immediately. 

For the last part of the proof, consider an arbitrary system  ${\bf X_T}=(X,\Sigma, \mu, T)$. Take  arbitrary functions $f,g,h\in L^p({\bf X_T})$ with $L^p$ norms equal to $a,b$ and $c$, respectively. Theorem ~\ref{thm:Halmos} guarantees the existence of a net of transformations $(\beta_s)\subset C(Y,X)$ such that $\beta_s S\beta_s^{-1}\to T$ in the weak topology. It is an easy verification that 
\begin{align*}
\sup_{N\ge 1}\frac{1}{N}&\sum_{n=1}^{N}f\circ\beta_s(S^n\beta_s^{-1}x)g\circ\beta_{s}(S^{2n}\beta_s^{-1}x)h\circ\beta_s(S^{3n}\beta_s^{-1}x)\\&=\sup_{N\ge 1}\frac{1}{N}\sum_{n=1}^{N}f((\beta_sS\beta_s^{-1})^nx)g((\beta_sS\beta_s^{-1})^{2n}x)h((\beta_sS\beta_s^{-1})^{3n}x)
\end{align*}
 for $\mu$ a.e. $x$. The functions $f\circ\beta_s$,  $g\circ\beta_s$ and $h\circ\beta_s$  have  the same norms $a,$ $b$ and $c$ respectively in $L^{p}({\bf Y_S})$ , hence
$$\mu\{x\in X:\sup_{N\ge 1}\frac{1}{N}\sum_{n=1}^{N}f((\beta_sS\beta_s^{-1})^nx)g((\beta_sS\beta_s^{-1})^{2n}x)h((\beta_sS\beta_s^{-1})^{3n}x)>\lambda\}\le\gamma({\bf Y_S},a,b,c,\lambda). $$ Since $f,g,h$ were arbitrary, we get that for each $s$, 
$$\gamma({\bf Y_{\beta_sS\beta_s^{-1}}},a,b,c,\lambda)\le\gamma({\bf Y_S},a,b,c,\lambda)$$
The fact that $S_{\gamma({\bf Y_S},a,b,c,\lambda)}$ is closed in $C(X)$ finishes the proof.

\begin{remark}
The obvious analog of the above Transference Priciple also holds, with no essential modifications in the proof, for each of the averages in \eqref{eq:averg}.
\end{remark}

\end{document}